\documentclass[11pt]{amsart}
\usepackage[utf8]{inputenc}
\usepackage[english]{babel}

\usepackage[square,sort,comma,numbers]{natbib}
\usepackage{pinlabel}

\usepackage[dvipsnames]{xcolor}
\usepackage{float}
\usepackage{pinlabel}
\usepackage{graphicx}
\usepackage{tikz}
\usetikzlibrary{arrows}
\usetikzlibrary{matrix,arrows,decorations.pathmorphing}

\usepackage{amsmath}
\usepackage{amssymb}
\usepackage{amsthm}

\usepackage[all]{xy}
\usepackage[mathcal]{eucal}
\usepackage{verbatim}\usepackage{fullpage} \usepackage{wrapfig}
\usepackage{lipsum}
\usepackage[all]{xy}
\theoremstyle{plain}
\newtheorem{thm}{Theorem}[section]

\newtheorem{lem}[thm]{Lemma}

\newtheorem{prop}[thm]{Proposition}

\newtheorem{claim}[thm]{Claim}

\newtheorem{qu}[thm]{Problem}
\theoremstyle{definition}
\newtheorem{defn}[thm]{Definition}

\theoremstyle{remark}
\newtheorem*{rem}{Remark}

\newcommand{\nc}{\newcommand}
\nc{\dmo}{\DeclareMathOperator}

\DeclareMathOperator{\Diff}{Diff}

\DeclareMathOperator{\interior}{Int}

\DeclareMathOperator{\Mod}{Mod}

\DeclareMathOperator{\Homeo}{Homeo}

\nc{\para}[1]{\medskip\noindent\textbf{#1.}}
\title{Non-realizability of the Torelli group as area-preserving homeomorphisms}

\address{\newline Department of Mathematics   \newline California Institute of Technology   \newline Pasadena, CA 91125,  USA}
\email{chenlei@caltech.edu, markovic@caltech.edu}

\author{Lei Chen and Vladimir Markovic}

\begin{document}
\maketitle
\begin{abstract}
Nielsen realization problem for the mapping class group $\Mod(S_g)$ asks whether the natural projection $p_g: \Homeo_+(S_g)\to \Mod(S_g)$ has a section. While all the previous results use torsion elements in an essential way, in this paper, we focus on the much more difficult problem of realization of torsion-free subgroups of $\Mod(S_g)$. The main result of this paper is that the Torelli group has no realization inside the area-preserving homeomorphisms. \end{abstract}

\section{Introduction}
Let $S_g$ be a surface of genus $g$. Let $p_g: \Homeo_+(S_g)\to \Mod(S_g)$ be the natural projection where $\Homeo_+(S_g)$ denotes the group of orientation-preserving homeomorphisms of $S_g$ and $\Mod(S_g):=\pi_0(\Homeo_+(S_g))$. In 2007, Markovic \cite{Mar} answered a well-known question of Thurston that $p_g$ has no section for $g\ge 5$. The proof in \cite{Mar} uses both torsions and the braid relations in an essential way, which both disappear in most finite index subgroups of $\Mod(S_g)$. Motivated by this, Farb \cite[Question 6.6]{Farb} asked the following question:
\begin{qu}[Sections over finite index subgroups]
Does the natural projection $p_g$ have a section over every finite index subgroup of $\Mod(S_g)$, or not?
\end{qu}
This problem presents two kinds of difficulties: the lack of understanding of finite index subgroups of $\Mod(S_g)$ and the lack of understanding of relations in $\Homeo_+(S_g)$. To illustrate the latter, we state the following problem (\cite{MT}[Problem 1.2]).
\begin{qu}
Give an example of a finitely-generated, torsion free group $\Gamma$, and a surface S, such that $\Gamma$ is not isomorphic to a subgroup of $\Homeo_+(S)$. 
\end{qu}
Motivated by the above problems and difficulties, we study the section problem for the Torelli group ${\mathcal I}(S_g)$, which is torsion free (e.g., \cite[Theorem 6.8]{FM}). Recall that ${\mathcal I}(S_g)$ is the subgroup of $\Mod(S_g)$ that acts trivially on $H_1(S_g;\mathbb{Z})$. For any area form on $S_g$, let $\Homeo_+^a(S_g)$ be the group of orientation-preserving, area-preserving homeomorphisms of $S_g$. In this paper, we prove the following:
\begin{thm}\label{main}
The Torelli group cannot be realized as a group of area-preserving homeomorphisms on $S_g$ for $g\ge 6$. In other words, the natural projection $p_g^a: \Homeo_+^a(S_g)\to \Mod(S_g)$ has no section over  ${\mathcal I}(S_g)$.
\end{thm} 
The property we use about the Torelli group is that it is generated by simple bounding pair maps \cite{Johnson}. To extend the method of this paper to study the Nielsen realization problem for all finite index subgroups of $\Mod(S_g)$, we need to study the subgroup generated by powers of simple bounding pair maps or powers of simple Dehn twists. In most cases, this subgroup is an infinite index subgroup of $\Mod(S_g)$ called the power group; see \cite{Funar} for discussions of power groups.

% On the other hand, our method could apply for all finite index subgroup if the following problem is solved. Let $A$ be a closed annulus and $\Homeo^a_+(A)$ be the group of orientation-preserving homeomorphisms of $A$. The translation number of an element $T\in \Homeo^a_+(A)$ is the difference of rotation numbers of $T$ on the two boundary components of $A$.
%\begin{qu}
%If $T\in \Homeo^a_+(A)$ has translation number $k$ for $k\neq 0$ in $A$, then can a power of $T$ be written as a product of commutators in its centralizer in $\Homeo^a_+(A)$?
%\end{qu}

\para{Previous work}
Nielsen posed the realization problem for finite subgroups of $\Mod(S_g)$ in 1943 and Kerckhoff \cite{Kerk} showed that a lift always exists for finite subgroups of $\Mod(S_g)$. The first result on Nielsen realization problem for the whole mapping class group is a theorem of Morita \cite{Mor} that there is no section for the projection $\Diff^2_+(S_g)\to \Mod(S_g)$ when $g\ge 18$. Then Markovic \cite{Mar} (further extended by Markovic--Saric \cite{MS} on the genus bound; see also \cite{LeCalvez} for simplification of the proof and \cite{lei} for the proof in the braid group case) showed that $p_g$ does not have a section for $g\ge 2$. Franks--Handel \cite{FH}, Bestvina--Church--Suoto \cite{BCS} and Salter--Tshishiku \cite{ST} also obtained non-realization theorems for $C^1$ diffeomorphisms. Notice that MoritAs result also extends to all finite index subgroups of $\Mod(S_g)$, but all the other results that we mention above do not extend to the case of finite index subgroups. We refer the readers to the survey paper by Mann--Tshishiku \cite{MT} for more history and previous ideas. 

We remark that the Nielsen realization problem for the Torelli group is also connected with another well-known MoritAs conjecture on the non-vanishing of the even MMM classes. Morita showed that most MMM classes vanish on $\Diff^2_+(S_g)$ and he conjectured that the even MMM classes do not vanish on the Torelli group \cite[Conjecture 3.4]{Moritasurvey}. Therefore, if one can prove MoritAs conjecture, one also gives a proof that the Torelli group cannot be realized in $\Diff^2_+(S_g)$.

\para{Ingredients of the paper}
The proof in this paper is essentially a local argument by considering the action on a sub-annulus. We use the following key ingredients:
\begin{enumerate}
\item Markovic's theory on minimal decomposition, extending it to the pseudo-Anosov case;
\item Poincar\'e-Birkhoff's theorem on existence of periodic orbits;
\item Handel's theorem on the closeness of the rotation interval.
\item Matsumoto's theorem about prime ends rotation numbers.
\end{enumerate}
Let $c$ be a separating simple closed curve and $T_c$ be the Dehn twist about $c$. The goal of the argument is to find an invariant subsurface with the frontier homotopic to $c$ such that the action of $T_c$ on the frontier has an irrational rotation number. Then by studying the action on the frontier, the fact that $T_c$ has an irrational rotation number is incompatible with the group structure. The main work of the paper is to obtain the invariant subsurface. This is done by using Poincar\'e-Birkhoff's theorem and Handel's theorem. Matsumoto's theorem is used to reduce the problem into a one-dimensional problem.

\para{Acknowledgements}  We in debt to the anonymous referee for many useful suggestions, extensive comments in particular for schooling us about rotation numbers. We also thank Danny Calegari for useful discussions. 

\vskip 2cm

\section{Rotation number of annulus homeomorphisms}
In this section, we discuss the properties of rotation numbers on annuli. 

\subsection{Rotation number of an area-preserving homeomorphism of an annulus}
Firstly, we define the rotation number for geometric annuli. Let 
\[
N=N(r)=\{w \in \mathbb{C}:\frac{1}{r}< |w|<r \}
\]
be the geometric annulus in the complex plane $\mathbb{C}$. Denote the geometric strip in $\mathbb{C}$ by 
\[
P=P(r)=\{x+iy=z\in \mathbb{C}:|y|<\frac{\log r}{2\pi}\}.
\] 
The map $\pi(z)=e^{2\pi iz}$ is a holomorphic covering map  $\pi: P\to N$. The deck transformation on $P$ is $T(x,y)=(x+1,y)$. 

Denote by $p_1:P\to \mathbb{R}$ the projection to the $x$-coordinate, and by $\Homeo_+(N)$ the group of homeomorphisms of $N$ that preserves orientation and the two ends.  Fix $f\in \Homeo_+(N)$, and $x\in N$, and let $\widetilde{x}\in P$ and $\widetilde{f}\in \Homeo_+(P)$ denote lifts of $x$ and $f$ respectively. We define the translation number of the lift $\widetilde{f}$ at $\widetilde{x}$ by
\begin{equation}\label{rot-0}
\rho(\widetilde{f},\widetilde{x},P)=\lim_{n\to \infty} (p_1(\widetilde{f}^n(\widetilde{x}))-p_1(\widetilde{x}))/n.
\end{equation}
The rotation number of $f$ at $x$ is then defined as
\begin{equation}\label{rot}
\rho(f,x,N)=\rho(\widetilde{f},\widetilde{x},P)  \,\,\,\,\,\, \,   \,\,\,\,\,\, \, (\text{mod 1}).
\end{equation}

The rotation number is not defined everywhere (see, e.g., \cite{Franks} for more background on rotation numbers). The closed annulus  $N_c$ is 
\[
N_c=\{\omega \in \mathbb{C}:\frac{1}{r}\le |\omega|\le r \},
\]
For $f\in \Homeo_+(N_c)$, the rotation and translation numbers are defined analogously.

Let $A$ be an open annulus embedded in a Riemann surface (in particular this endows $A$ with the complex structure). By the Riemann mapping theorem, there is a unique $N(r)=N$ and a conformal map $u_A: A\to N$. For any $f\in \Homeo_+(A)$ (the group of end-preserving homeomorphisms), we 
define the rotation number of $f$ on $A$ by
\[\rho(f,x,A):=\rho(g,u_A(x), N),\]  where $g=u_A\circ f\circ u_A^{-1}$.

We have the following theorems of Poincar\'e-Birkhoff and Handel about rotation numbers \cite{Handel} (See also Franks \cite{Franks}). 
\begin{thm}[Properties of rotation numbers]\label{rotationproperty}
If $f: N_c \to N_c$ is an orientation preserving, boundary component preserving, area-preserving homeomorphism and $\widetilde{f}: P_c \to P_c$ is any lift, then:
\begin{itemize}
\item (Handel) The translation set \[
R(\widetilde{f})=\bigcup_{\widetilde{x} \in P_c} \rho(\widetilde{f},\widetilde{x},P_c)\] is a closed interval.
\item (Poincar\'e-Birkhoff) If $r\in R(\widetilde{f})$ is rational, then there exists a periodic orbit of $f$ realizing the rotation number $r$ mod $1$. 
\end{itemize}
\end{thm}

\subsection{Separators and its property}
We let $A$ continue to denote an open annulus embedded in a Riemann surface. Then $A$ has two ends and we choose  one of them to be the left end and the other one to be the right end. We call a  subset $X\subset \interior(A)$ \emph{separating} (or essential) if every 
arc $\gamma \subset A$ which connects the two ends of $A$  must intersect $X$. 

\begin{defn}[Separator]
We call a subset $M\subset A$ a \emph{separator} if $M$ is compact, connected and separating.
\end{defn}
The complement of $M$ in $A$ is a disjoint union of open sets. We have the following lemma.
\begin{lem}\label{separator}
Let $M$ be a separator. Then there are exactly two connected components $A_L(M)$ and $A_R(M)$ of $A-M$ which are open annuli homotopic to $A$ and with the property that $A_L(M)$ contains the left end of $A$ and $A_R(M)$ contains the right end of $A$. All other components of $A-M$ are simply connected.
\end{lem}
\begin{proof}
We compactify the annulus $A$ by adding  points $p_L$ and $p_R$ to the corresponding ends of $A$. The compactifications is a two sphere $S^2$. Moreover, $M$ is a compact and connected subset of $S^2 -\{p_L,p_R\}$.

Now, we observe that every component of  $S^2-M$ is  simply connected. Denote by $\Omega_L$ and $\Omega_R$ the  connected components of $S^2-M$  
containing $p_L$ and $p_R$  respectively. Since $M$ is separating we conclude that these are two different components. We define $A_L(M)=\Omega_L-p_L$ and $A_R(M)=\Omega_R-p_R$. It is easy to verify that these are required annuli.
\end{proof}

We now prove another property of a separator. Let $\pi: \widetilde{A}\to A$ be the universal cover.

\begin{prop}\label{sepprop}
Let $M \subset A$ be  a separator. Then $\pi^{-1}(M)$ is connected.
\end{prop}

\begin{proof}
Let $M_n \subset A$ be a decreasing sequence of separators such that  each $M_n$ is a compact domain with smooth boundary, and
$$
\bigcap\, M_n=M.
$$
(It is elementary to construct such $M_n$'s). Then 
$$
\bigcap\, \pi^{-1}(M_n)=\pi^{-1}(M),
$$
and $\pi^{-1}(M_n)$ is  decreasing. If  each $\pi^{-1}(M_n)$ is connected then  $\pi^{-1}(M)$ is the intersection of a decreasing sequence of connected sets, and  
it is connected as such.  Therefore, it suffices to prove  that $\pi^{-1}(M)$ is connected assuming  $M$ is a separator which  is a compact domain with smooth boundary. We do this in the remainder of the proof.

Since $M$ is a compact domain with boundary which separates the two ends of $A$, we can find a circle $\gamma\subset M$ which is essential in $A$ (i.e. $\gamma$ is a separator itself) (note that $M$ has only finitely many boundary components).  Denote by $T$ the deck transformation of $\widetilde{A}$. Thus, the lift $\pi^{-1}(\gamma)$ is a $T$ invariant, connected subset of $\widetilde{A}$. Let $C$ be the component of $\pi^{-1}(M)$ which contains $\pi^{-1}(\gamma)$. Then $C$ is $T$ invariant. We show $\pi^{-1}(M)=C$.

Let $p\in M$. Since $M$ is a compact  domain with smooth boundary, we can find an embedded closed arc $\alpha\subset M$ which connects $p$ and $\gamma$. Let $\widetilde{p}$ be a lift of $p$ and let $\widetilde{\alpha}$ be the corresponding lift of $\alpha$ such that $\widetilde{p}$ is one of its endpoints. Then, the other endpoint of  $\widetilde{\alpha}$ is in $\pi^{-1}(\gamma)$, and this shows that  $\widetilde{p} \in C$. This concludes the proof.

\end{proof}

Now we discuss an ordering on the set of separators. 
\begin{prop}\label{ordering}
Suppose $M_1,M_2 \subset A$ are two disjoint separators. Then either $M_1\subset A_L(M_2)$  or  $M_1\subset A_R(M_2)$. Moreover, $M_1\subset A_L(M_2)$ implies  $M_2\subset A_R(M_1)$. 
\end{prop}
\begin{proof}
Since $M_1$ is connected it follows that $M_1$ is a subset of a connected component $C$ of $A-M_2$. If $C$ is simply connected, the cover $\pi^{-1}(C)\to C$ is a trivial cover. Let $\widetilde{C}$ be a connected component of $\pi^{-1}(C)$. By  Proposition \ref{sepprop}, the set $\pi^{-1}(M)$ is connected so it is contained in a single connected component of $\pi^{-1}(C)$. However, this contradicts the fact that  $\pi^{-1}(M)$ is translation invariant. Thus,  either $M_1\subset A_L(M_2)$ or $M_1\subset A_R(M_2)$.

Suppose  $M_1\subset A_L(M_2)$. Then $A_L(M_1)\subset A_L(M_2)$ as well. On the other hand, by the first part of the proposition we already know that either $M_2\subset A_L(M_1)$ or $M_2\subset A_R(M_1)$. If  $M_2\subset A_L(M_1)$, then $A_L(M_2)\subset A_L(M_1)$. This shows that $A_L(M_1)\subset A_L(M_2)$ which implies that $M_2\subset A_L(M_2)$. This is absurd so we must have $M_2\subset A_R(M_1)$.
\end{proof}

\vskip .3cm

\begin{defn}\label{def-ordering}

The inclusion $M_1\subset A_L(M_2)$ is denoted as $M_1<M_2$.

\end{defn}

\vskip .3cm

\subsection{The rotation interval of an annular continuum and prime ends}

Let  $K\subset A$ be a separator (in literature also known as an \emph{essential continuum}). We call $K$ an essential \emph{annular continuum} if $A - K$ has exactly two components. Observe that an essential annular continuum can be expressed  as a decreasing intersection of essential closed topological annuli in $A$.

It is possible to turn any separator $M\subset A$ into  an essential  annular continuum. Let $M$ be a separating connected set. By Lemma \ref{separator}, we know that $A-M$ has exactly two connected annular components $A_L(M)$ and $A_R(M)$, and all other components of $A-M$ are simply connected. We call a simply connected component of $A-M$ a \emph{bubble component}. Then the \emph{annular completion} $K(M)$ of $M$ is defined as the union of $M$ and the corresponding  bubble components of $A-M$.

\begin{prop}
Let $M\subset A$ be a separator. Then the annular competition $K(M)$ is an annular continuum. 
\end{prop}
\begin{proof}
We can again compactify $A$ by adding the points $p_L$ and $p_R$, one at each end. The compactification is the two sphere $S^2$. Then $A_L(M)$ and $A_R(M)$ are two disjoint open discs in $S^2$, and $K(M)=S^2-(A_L(M)\cup A_R(M))$. But the complement of two disjoint open discs in $S^2$ is connected. This proves the proposition.

\end{proof}

Now let $f$ be a homeomorphism of $A$ that leaves an annular continuum $K$ invariant. If $\mu$ is an invariant Borel probability measure,  we define the $\mu$-rotation number 
\[
\sigma(f, \mu)=\int_A\phi d\mu
\]
where $\phi: A \to \mathbb{R}$ is the function which lifts to the function $p_1\circ f-p_1$ on $\widetilde{A}$ (recall that $p_1:\widetilde{A} \to  \mathbb{R}$ is the projection onto the first coordinate).

The set of $f$ invariant Borel probability measures on $K$ is a non empty, convex, and compact set (with respect to the weak topology on the space of measures).
We define the \emph{rotation interval} of $K$
\[
\sigma(f, K)=\{\sigma(f, \mu)|\mu\in M(K)\}
\]
which is a non-empty segment $[\alpha,\beta]$ of $\mathbb{R}$. The interval is non empty because there exists at least one $f$ invariant measure, and it is an interval because the set of $f$ invariant measures is convex.

The following is a classical result of Franks--Le Calvez \cite[Corollary 3.1]{FLC}.
\begin{prop}
If $\sigma(f, K)=\{\alpha\}$, the sequence 
\[
\frac{p_1\circ f^n(x)-p_1(x)}{n}
\]
converges uniformly for $x\in \pi^{-1}(K)$ to the constant function $\alpha$. This implies that points in $K$ all have the rotation number $\alpha$.
\end{prop}

The following  theorem of Franks--Le Calvez \cite[Proposition 5.4]{FLC}  is a generalization of the Poincar\'e-Birkhoff Theorem.
\begin{thm}\label{PB}
If $f$ is area-preserving and $K$ is an annular continuum, then every rational number in $\sigma(f,K)$ is realized by a periodic point in $K$.
\end{thm}

The theory of prime ends is an important tool in the study of 2-dimensional dynamics which can be used to transform a 2-dimensional problem into a 1-dimensional problem. Recall that we assume that $A$ is an open annulus embedded in a Riemann surface $S$.  Suppose that $f$ is a homeomorphism of $S$ which leaves $A$ invariant. Furthermore, let $K\subset A$ be an annular continuum and  suppose that $f$  leaves $K$ invariant. Then both $A_L(K)$ and $A_R(K)$ are $f$ invariant. 

Since $A$ is embedded in $S$, we can define the frontiers of $A$, $A_L(K)$, and $A_R(K)$. By Carath\'eodory's theory of prime ends (see, e.g., \cite[Chapter 15]{Milnor}), the homeomorphism $f$   yields an action on the frontiers of $A_L(K)$ and $A_R(K)$.  Consider the right hand frontier of $A_L(K)$ (the one which is contained in $A$). Then the set of prime ends on this frontier is homeomorphic to the circle, and we denote by  $f_L$ the induced homeomorphism this circle. Likewise, the set of prime ends on left hand  frontier of $A_R(K)$  is homeomorphic to the circle, and we denote by  $f_R$ the induced homeomorphism this circle. 

The rotation number of a circle homeomorphism   (defined by Equation \eqref{rot}), is well defined everywhere and is the same number for any point on the circle. The rotation numbers of $f_L$ and $f_R$ are called $r_L$ and $r_R$. We refer to them as the left and right prime end rotation numbers of $f$. We have the following theorem of Matsumoto \cite{Mat}.
\begin{thm}[Matsumoto's theorem]\label{Matsumoto}
If $K$ is an annular continuum, then its left and right
prime ends rotation numbers $r_L,r_R$ belong to the rotation interval $\sigma(f, K)$.
\end{thm}

\section{Minimal decompositions and the Torelli group theory}

\subsection{Minimal decompositions} We recall the theory of minimal decompositions of surface homeomorphisms. This is established in \cite{Mar}. 
Firstly we recall the upper semi-continuous decomposition of a surface; see also Markovic \cite[Definition 2.1]{Mar}. Let $M$ be a surface. 
\begin{defn}[Upper semi-continuous decomposition]
Let $\mathbf{S}$ be a collection of closed, connected subsets of $M$. We say that $\mathbf{S}$ is an upper semi-continuous decomposition of $M$ if the following holds:
\begin{itemize}
\item If $S_1,S_2\in \mathbf{S}$, then $S_1\cap S_2=\emptyset$.
\item If $S\in \mathbf{S}$, then $E$ does not separate $M$; i.e., $M-S$ is connected.
\item We have $M=\bigcup_{S\in \mathbf{S}} S$.
\item If $S_n\in \mathbf{S}, n\in \mathbb{N}$ is a sequence that has the Hausdorff limit equal to $S_0$ then there exists $S\in\mathbf{S}$ such that $S_0\subset S$.
\end{itemize}
\end{defn}
Now we define acyclic sets on a surface.
\begin{defn}[Acyclic sets]
Let $S\subset M$ be a closed, connected subset of $M$ which does not separate $M$. We say that $S$ is \emph{acyclic} if there is a simply connected open set $U\subset M$ such that $S \subset U$ and $U-S$ is homeomorphic to an annulus.
\end{defn}
The simplest examples of acyclic sets are a point, an embedded closed arc and an embedded closed disk in $M$. Let $S\subset M$ be a closed, connected set that does not separate M. Then $S$ is acyclic if and only if there is a lift of $S$ to the universal cover $\widetilde{M}$ of $M$, which is a compact subset of $\widetilde{M}$. 
The following theorem is a classical result called Moore's theorem; see, e.g., \cite[Theorem 2.1]{Mar}. 
\begin{thm}[Moore's theorem]\label{moore}
Let $M$ be a surface and $\mathbf{S}$ be an upper semi-continuous decomposition of $M$ so that every element of $\mathbf{S}$ is acyclic. Then there is a continuous map $\phi:M\to M$ that is homotopic to the identity map on $M$ and such that for every $p\in M$, we have $\phi^{-1}(p)\in \mathbf{S}$. Moreover  $\mathbf{S}=\{\phi^{-1}(p)|p\in M\}$.
\end{thm}
We call the map $M\to M/\sim$ the \emph{Moore map} where $x\sim y$ if and only if $x,y\in S$ for some $S\in \mathbf{S}$. The following definition is \cite[Definition 3.1]{Mar}
\begin{defn}[Admissible decomposition]
Let $\mathbf{S}$ be an upper semi-continuous decomposition of $M$. Let $G$ be a subgroup of $\Homeo(M)$. We say that $\mathbf{S}$ is admissible for the group $G$ if the following holds:
\begin{itemize}
\item Each $f\in G$ preserves setwise every element of $\mathbf{S}$.
\item Let $S\in \mathbf{S}$. Then every point, in every frontier component of the surface $M-S$ is a limit of points from $M-S$ which belong to acyclic elements of $\mathbf{S}$. 
\end{itemize}
If $G$ is a cyclic group generated by a homeomorphism $f:M\to M$ we say that $\mathbf{S}$ is an admissible decomposition of $f$.
\end{defn}
An admissible decomposition for $G<\Homeo(M)$ is called \emph{minimal} if it is contained in every admissible decomposition for $G$. We have the following theorem \cite[Theorem 3.1]{Mar}.
\begin{thm}[Existence of minimal decompositions]
Every group $G<\Homeo(M)$ has a unique minimal decomposition.
\end{thm}
Denote by $\mathbf{A}(G)$ the sub collection of acyclic sets from $\mathbf{S}(G)$. By a mild abuse of notation, we occasionally refer to  $\mathbf{A}(G)$ as a subset of $S_g$ (the union of all sets from $\mathbf{A}(G)$). To distinguish the two notions we do the following. When we refer to $\mathbf{A}(G)$ as a collection then we consider it as the collection of acyclic sets. When we refer to as a set (or a subsurface of $S_g$) we have in mind the other meaning.

We have the following result \cite[Proposition 2.1]{Mar}.
\begin{prop}\label{subsurface}
Every connected component of $\mathbf{A}(G)$ (as a subset of $S_g$)   is a proper subsurface of $M$ with finitely many ends. 
\end{prop}
\begin{lem}\label{finer}
For $H<G<\Homeo(M)$, we have that  $\mathbf{A}(G)\subset \mathbf{A}(H)$.\end{lem}
\begin{proof}  $\mathbf{A}(G)\subset \mathbf{A}(H)$ because the minimal decomposition of $G$ is also an admissible decomposition of $H$ and the minimal decomposition of $H$ is finer than that of $G$. 
\end{proof}

\subsection{The Torelli group and simple BP maps} 

From this point one, we fix a closed surface $S_g$ and  a separating simple closed curve $c$ which divides the surface $S_g$  into a genus $k$ subsurface and a genus $g-k$ subsurface for $k>2$. We call the genus $k$ part the \emph{left subsurface} $S_L$ and the genus $g-k$ part the \emph{right subsurface} $S_R$. 
For any simple closed curve $a$, denote by $T_a$ the Dehn twist about $a$. 
\begin{defn}
For $a,b$ two disjoint non-separating curves on $S_L$ bounding a genus $1$ subsurface, we call the bounding pair map $T_aT_b^{-1}$ a \emph{simple bounding pair map}, which is shortened as a simple BP map. 
\end{defn}
Let ${\mathcal L\mathcal I}(c)\subset {\mathcal I}(S_g)$ be the subgroup generated by simple BP maps on the left subsurface $S_L$. About ${\mathcal L\mathcal I}(c)$, we have the following proposition.
\begin{prop}\label{Johnson}
We have that $T_c^{2-2k}\in {\mathcal L\mathcal I}(c)$ and $T_c^{2-2k}$ is a product of commutators in ${\mathcal L\mathcal I}(c)$. 
\end{prop}
\begin{proof}
The Birman exact sequence for the mapping class group of $S_L$ fixing the boundary component has the following form
\[
1\to \pi_1(UTS_{k})\xrightarrow{\text{Push}} \Mod(S_{k}^1)\to \Mod(S_{k})\to 1.
\]
Here, $UTS_{k}$ denotes the \emph{unit tangent bundle} of $S_{k}$; i.e., the $S^1$-subbundle of the tangent bundle $TS_{k}$ consisting of unit-length tangent vectors (relative to an arbitrarily-chosen Riemannian metric). In this context, the kernel $\pi_1(UTS_{k})$ is known as the \emph{disk-pushing subgroup}. Let $e$ be the generator of the center of $\pi_1(UTS_{k})$, which satisfies that $T_c=\text{Push}(e)$. We have the following $\mathbb{Z}$-extension
\begin{equation}\label{extension}
1\to \mathbb{Z}\to \pi_1(UTS_{k})\to \pi_1(S_k)\to 1.
\end{equation}
For every central $\mathbb{Z}$-extension
\[
1\to \mathbb{Z}\to \widetilde{Q} \to Q\to 1,
\] there is an associated Euler number which is evaluated on
$H_2(Q;\mathbb{Z})$. The Euler number of \eqref{extension} is $2-2k$ on the generator of $H_2(\pi_1(S_k);\mathbb{Z})$. This means that for a standard generating set $a_1,b_1,...,a_{k},b_{k}$ of $\pi_1(S_k)$, their lifts $\widetilde{a_1},\widetilde{b_1},...,\widetilde{a_{k}},\widetilde{b_{k}}$ in $\pi_1(UTS_{k})$  satisfies the following:
\[
e^{2-2k}=[\widetilde{a_1},\widetilde{b_1}]...[\widetilde{a_{k}},\widetilde{b_{k}}].
\]
Therefore we have the relation
\[
T_c^{2-2k}=[\text{Push}(\widetilde{a_1}),\text{Push}(\widetilde{b_1})]...[\text{Push}(\widetilde{a_{k}}),\text{Push}(\widetilde{b_{k}})].
\]
Up to multiplying a power of $T_c$, the map $\text{Push}(\widetilde{a_i})$ or $\text{Push}(\widetilde{b_i})$ is a single BP map (see \cite[Fact 4.7]{FM}). Any BP map $T_{a}T_b^{-1}$ on $S_L$ is a product of simple BP maps since we can find simple closed curves $c_0=a,...,c_{k+1}=b$ such that $c_i,c_{i+1}$ bounds a genus $1$ subsurface. Then 
\[
T_{a}T_b^{-1}=\Pi_{i=0}^{k}T_{c_i}T_{c_{i+1}}^{-1}\]
Thus $T_c^{2-2k}$ can be written as a product of simple BP maps. 
\end{proof}
In this paper, we choose $k=4$ for the rest of the paper. The reason for this is that we need some room for the existence of pseudo-Anosov Torelli elements in ${\mathcal L\mathcal I}(c)$, which is essential for applying the minimal decomposition theory.

\section{Characteristic annuli and Rotation numbers}

\subsection{Minimal decomposition for a realization}
From now on, we work with the assumption that there exists a realization of the Torelli group 
\[
\mathcal{E}: {\mathcal I}(S_g)\to \Homeo^a_+(S_g).
\]
For an element $f\in \mathcal{I}(S_g)$, or a subgroup $F<\mathcal{I}(S_g)$, we shorten $\mathbf{A}(\mathcal{E}(f))$ as $\mathbf{A}(f)$, and $\mathbf{A}(\mathcal{E}(F))$ as $\mathbf{A}(F)$, to denote the corresponding collections of acyclic components. Recall that $c\subset S_g$ is a fixed simple closed curve that divides $S_g$ into subsurfaces $S_L$ and $S_R$ so that $S_L$ has genus $4$ (see the definition in the previous section). We have the following theorem about the minimal decompositions of $\mathcal{E}(T_c^{-6})$.

\vskip .3cm

\begin{thm}\label{minimal}
The set $\mathbf{A}(T_c^{-6})$ has a component $\mathbf{L}(c)$ which is homotopic to $S_L$ and a component $\mathbf{R}(c)$ homotopic to $S_R$.
\end{thm}

\vskip .3cm

\begin{rem}We use the same argument as in \cite{Mar}. Since we are working with the Torelli group which contains no Anosov elements, we need to use pseudo-Anosov elements. The argument is almost the same as \cite{Mar}. For this reason, we postpone the proofs to Section 6. 
\end{rem}

\vskip .3cm

\subsection{Invariant annuli}

In the remainder of the paper we let
\[
\mathbf{B}=S_g-\mathbf{L}(c)-\mathbf{R}(c) .
\]
Since $\mathbf{L}(c)$ and $\mathbf{R}(c)$ are open (as subset of $S_g$), it follows that $\mathbf{B}$  is compact. Moreover,  $\mathbf{L}(c)$ and $\mathbf{R}(c)$ are disjoint, each having exactly one end (and this end is homotopic to $c$),  so it follows that  $\mathbf{B}$ is connected. By definition, we know that $\mathbf{B}$ is ${\mathcal E}({\mathcal L\mathcal I}(c))$-invariant since homeomorphisms from ${\mathcal E}({\mathcal L\mathcal I}(c))$ commute with $f$.

\vskip .3cm

To simplify the notation,  we set 
\[
f=\mathcal{E}(T_c^{-6}).
\]

\vskip .3cm

\begin{defn} We say that  $A\subset S_g$ is an invariant annulus if 
\begin{enumerate}
\item $A$ is an open annulus, homotopic to the curve $c$, 
\item $\mathbf{B}\subset A$,
\item $A$ is invariant under $f$.
\end{enumerate}
\end{defn}

\vskip .3cm

Next, we prove the  lemma which says  that the rotation numbers of points from $\mathbf{B}$ (under the action of $f$) do not depend on which invariant the annulus we use. 

\begin{lem}\label{notdepend}
Let $A_1,A_2$ be two invariant annuli. Then
\[
\rho(f,x,A_1)=\rho(f,x,A_2), \,\,\,\,\,\,\,\,\, x\in \mathbf{B}.
\]
\end{lem}

\begin{proof} We let $\pi_i:P_i\to A_i$, $i=1,2$,  denote the universal cover, where $P_i$ is the infinite strip in the complex plane such that  $A_i$ is (as a Riemann surface) isomorphic to $P_i / \langle T \rangle$, where $T(x,y)=(x+1,y)$. The height of the  strip $P_i$ depends on the modulus of $A_i\subset S_g$. We let 
$$
\mathbf{B}_i=\pi^{-1}_i(\mathbf{B}).
$$

\vskip .3cm

Since $A_1$ and $A_2$ are  open annuli containing the compact set $\mathbf{B}$, there is a homeomorphism 
$g: A_1\to A_2$, such that $g|_\mathbf{B}=\text{Id}$. Choose a lift  $\widetilde{g}:P_1\to P_2$ of $g$. Then  $\widetilde{g}(\mathbf{B}_1)=\mathbf{B}_2$. Moreover, since both $P_1$ and $P_2$ live in the same complex plane, and  have the same group of deck transformations, and since $\widetilde{g}$ conjugates the deck transformation to itself, it follows that

\vskip .3cm

\begin{equation}\label{eq-1}
d\big(y,\widetilde{g}(y)\big)<d_0, \,\,\,\, \text{for every}\,\,\,\,\, y\in \mathbf{B}_1,
\end{equation}
for some constant $d_0>0$ (here $d$ stands for the Euclidean distance on $P_1$).

\vskip .3cm

Let $\widetilde{f}_1:P_1\to P_1$ be a lift of $f$ to $P_1$. We then choose $\widetilde{f}_2:P_2\to P_2$, a lift of $f$ to $P_2$, such that 
\begin{equation}\label{eq-2}
\widetilde{f}_2=\widetilde{g}\circ \widetilde{f}_1\circ  \widetilde{g}^{-1}, \,\,\,\,\,\,\,\, \text{on}\,\,\,\,\,\,\,\,\, \mathbf{B}_2.
\end{equation}

Recall the definition of  the rotation number of $f$ at $x\in A_i$
$$
\rho(f, x ,A_i)=\lim_{n\to \infty} (p_1(\widetilde{f}^n_i(\widetilde{x}_i))-p_1(\widetilde{x}_i))/n  \,\,\,\,\,\, \,   \,\,\,\,\,\, \, (\text{mod 1}),
$$
where $\widetilde{x}_i$ is a lift of $x$ to $P_i$. Replacing (\ref{eq-1}) and (\ref{eq-2}) into this definition shows that 
$\rho(f, x ,A_1)=\rho(f, x ,A_2)$. The lemma is proved.

\end{proof}

\subsection{Characteristic annuli and rotation numbers}

Let $p_L: \mathbf{L}(c)\to \mathbf{L}(c)/\sim$ and $p_R: \mathbf{R}(c)\to \mathbf{R}(c)/\sim$ be the Moore maps of $\mathbf{L}(c)$ and $\mathbf{R}(c)$ corresponding to the decomposition $\mathbf{S}(c)$. Let $\mathbf{L}\subset \mathbf{L}(c)/\sim$ be an open annulus bounded by the end of $\mathbf{L}(c)'$ on one side, and by a simple closed curve on the other. The open annulus $\mathbf{R}\subset \mathbf{R}(c)/\sim$ is defined similarly. We have the following definition (see \cite[Chapter 5]{Mar}). 

\begin{defn} An annulus of the form $A=p_L^{-1}(\mathbf{L})\cup \mathbf{B} \cup p_R^{-1}(\mathbf{R})$ is called  a \emph{characteristic annulus}.
\end{defn}

\vskip .3cm

Every characteristic annulus is an invariant annulus. We observe that $\mathbf{B}$ is a separator in $A$, that is, $\mathbf{B}$ is an essential, compact, and connected subset of $A$.  Note that a characteristic annulus $A$ is invariant under $f$, but it may not be invariant under homeomorphisms which are lifts 
(with respect to $\mathcal{E}$) of other elements from the Torelli group.  However,  $\mathbf{B}$ is invariant under these lifts  of elements from  ${\mathcal L\mathcal I}(c)\subset {\mathcal I}(S_g)$ (the subgroup generated by simple BP maps on the left subsurface $S_L$). As we see from the next lemma, the dynamical information about $f$ is contained in $\mathbf{B}$.

\vskip .3cm

\begin{lem}\label{gap} Fix a characteristic annulus $A$. Then
\begin{enumerate}
\item  every number $0<r<1$ appears as the rotation number $\rho(f,x,A)$, for some $x\in A$,
\item if  $0<\rho(f,x,A)<1$, then $x\in B$.
\end{enumerate}
\end{lem}

\begin{proof}

The idea is show that the translation numbers of the restrictions of $f$ to the frontiers of $A$ differ by $6$. We then apply Handel's theorem.

Define a new upper-semicontinuous decomposition $\mathbf{S}_{new}$ of $S_g$ as follows. Outside of $A$, the decomposition consists of elements of 
$\mathbf{S}(c)$ (note that the outside of $A$ is contained in $\mathbf{A}(c)$). Inside of $A$ the decomposition $\mathbf{S}_{new}$ consists of points. By definition, this is an upper-semicontinuous decomposition which consists of acyclic components only.

Let $p: S_g\to S_g':=S_g/\sim$ be the Moore map of $\mathbf{S}_{new}$ as in Figure \ref{Moore}. By definition, $p|_{A}$ is a homeomorphism. The action of $f$ on $S_g$ is semi-conjugated by $p$ to a homeomorphism $f'$ on $S_g'$. Since $f$ preserves each components of $\mathbf{S}(c)$, we know that $f'|_{S_g'-p(A)}=id$. The action $f'$ on $p(A)$ is conjugate to $f$ on $A$ since $p|_{A}$ is a homeomorphism. Since $p|_{\mathbf{L}(c)}=p_L|_{\mathbf{L}(c)}$ and $p|_{\mathbf{R}(c)}=p_R|_{\mathbf{R}(c)}$, the boundary components of $p(A)$ in $S_g'$ are two simple closed curves  $\partial_L$ and $\partial_R$.

\begin{figure}[h]
\minipage{0.48\textwidth}
  \includegraphics[width=\linewidth]{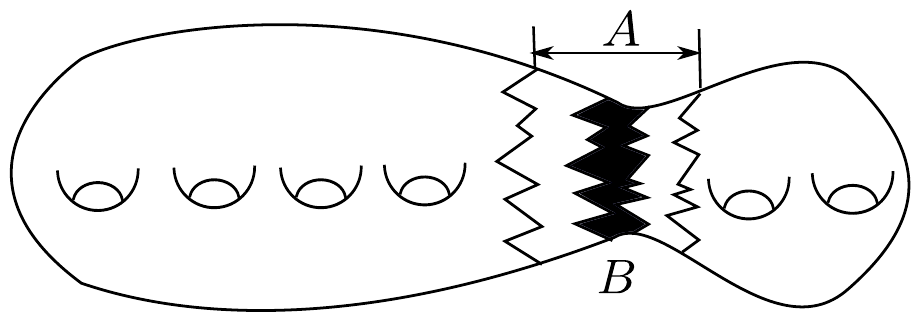}
\endminipage\hfill
  $\xrightarrow{p}$
\minipage{0.48\textwidth}
  \includegraphics[width=\linewidth]{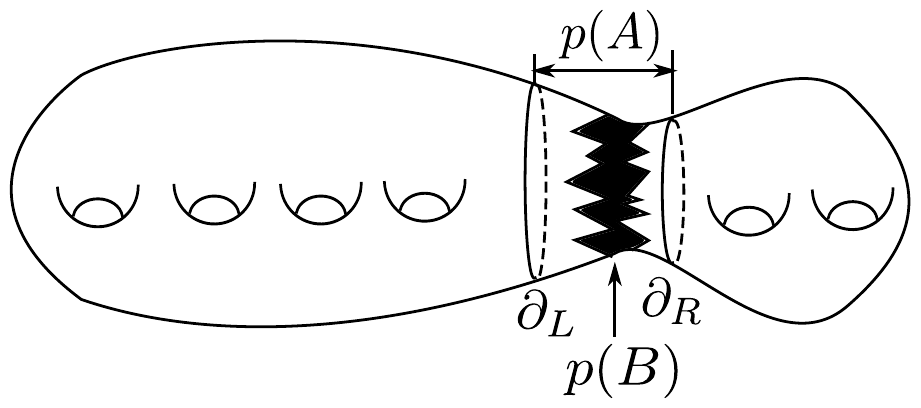}
\endminipage\hfill
\caption{the Moore map $p$ for $\mathbf{S}_{new}$}
\label{Moore}
\end{figure}

Recall that $\mathcal{E}$ is a realization. Thus, $f'$  is homotopic  to the standard Dehn twist map $T^{-6}_c$ on $S'_g$. Since $f'$ is identity map outside of $A'$,  the homeomorphism  $f'$ is homotopic to the standard Dehn twist map $T_c^{-6}$ on $A'$. By Handel's Theorem \ref{rotationproperty}, the collection of all translation numbers $\rho(\widetilde{f'},x,\widetilde{A'}_c)$ is a closed interval (here $\widetilde{f'}:P_c\to P_c$ is any lift of $f'$ to the infinite trip $P_c$ which is the universal cover of $A'$). Since this set contains points $0$ and $-6$, we conclude that every $0<r<1$ appears as the rotation number $\rho(f',x,A')$ for some $x\in A'$.  The statement (1) is proved.

For (2), recall that $A-\mathbf{B} \subset \mathbf{A}(c)$.  Let $C(x) \in \mathbf{A}(c)$ be the corresponding acyclic set that contains $x \in A-\mathbf{B}$. Denote by $\pi:P\to A$ the universal cover. Since $C(x)$ is acyclic we conclude that any connected component of  $\pi^{-1}\big(C(x)\big)$  is a compact set. Let $\widetilde{f}:P \to P$ be a lift of $f$. Since $f$ preserves the set $C(x)$, the homeomorphism  $\widetilde{f}$ permutes the connected components of $\pi^{-1}\big(C(x)\big)$. Therefore, the translation  number of $\widetilde{f}$ at points in $\pi^{-1}\big(C(x)\big)$ must be an integer (if it exists).
Therefore, the rotation number of $f$ at $x$ (if it exists) is $0$. This implies that $E_r\subset \mathbf{B}$.
\end{proof}

\vskip .3cm

\subsection{A special characteristic annulus $A_h$} \label{ahbh}
In this subsection, we define a special characteristic annulus $A_h$ with respect to a BP map $h$ and study its properties. Firstly, we have the following theorem about minimal decompositions of $\mathcal{E}(h)$ and $\mathcal{E}(\langle T_c^{-6},h\rangle)$ which will be proved in Section 6. Here $\langle T_c^{-6},h\rangle$ denotes the group generated by these two elements.

\begin{thm}\label{minimal-1}
For a simple BP map $h=T_aT_b^{-1}$, the set $\mathbf{A}(h)$ contains a component $\mathbf{R}(h)$ with two ends homotopic to $a,b$ respectively. Further more, the set $\mathbf{A}(\langle T_c^{-6},h\rangle)$ has a component $\mathbf{M}_1(c,h)$ with ends homotopic  to $a,b,c$ respectively and a component $\mathbf{M}_2(c,h)$ with ends homotopic  to $c$. 
\end{thm}
\begin{figure}[H]
 \includegraphics[scale=1]{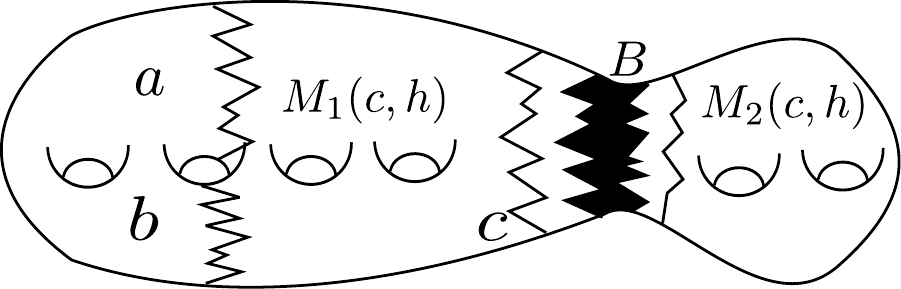}
 \caption{Location of $\mathbf{M}_1(c,h)$ and $\mathbf{M}_2(c,h)$}
\end{figure}

Let $\mathbf{M}(c,h)\subset S_g$ be the connected subsurface which contains $\mathbf{M}_1(c,h)$, which has  two ends, and which shares its two ends  
with $\mathbf{M}_1(c,h)$ (these are the two ends of $\mathbf{M}_1(c,h)$ homotopic to $a$ and $b$ respectively). 
We let \[
\mathbf{B_h}=\mathbf{M}(c,h)-\mathbf{M}_1(c,h)-\mathbf{M}_2(c,h).
\]

Since the decomposition of $f$ is finer than that of $\mathcal{E}(\langle T_c^{-6},h\rangle)$, we know that $\mathbf{M}_1(c,h)\subset \mathbf{L}(c)$ and $\mathbf{M}_1(c,h)\subset \mathbf{R}(c)$. This also implies that $\mathbf{B}\subset \mathbf{B_h}$. Similarly to how we defined the Moore map $p$ above, we define the Moore map $p_h:S_g\to S_g''=S_g/\sim$, with respect to  the new upper-semi continuous decomposition of $S_g$ defined as follows: on $\mathbf{M}_1(c,h) \cup \mathbf{M}_2(c,h)$ we use the (acyclic) components of  $\mathbf{S}(c)$, and on the rest of the surface $S_g$  each point is one component. We let   $A''$ be any sub-annulus of $p_h\big(\mathbf{M}(c,h)\big)$, bounded by two simple closed curve curves, and which  contains $p_h\big(\mathbf{B}_h\big)$.

Set \[A_h=p_h^{-1}(A'').\]

The minimal decomposition for $f$ is finer than the one of $\mathcal{E}(\langle T_c^{-6},h\rangle)$. The following claim and proposition follow from this observation.

\begin{claim} Each $A_h$  is a characteristic annulus. 
\end{claim}
\begin{rem}
We call $A_h$ a special characteristic annulus with respect to $h$.
\end{rem}

\begin{proof} The quotient map $p:S_g\to S_g'$ factors through the quotient map $p_h:S_g\to S_g''$. That is, $p=\xi\circ p_h$, where $\xi:S_g''\to S_g'$ is another quotient map. Each $A_h$ is given by $A_h=p_h^{-1}(A'')$. We let $A'=\xi(A'')$, and observe that  $A_h=p^{-1}(A')$. This proves the claim.
\end{proof}

\vskip .3cm

\begin{prop}\label{prop-novo} Fix a simple BP map $h$, and a special characteristic annulus $A_h$.  Then

\begin{enumerate}

\item  $\mathbf{B}\subset \mathbf{B_h}$,

\vskip .3cm

\item $\mathbf{B_h}\subset \mathbf{A}(h)$,

\vskip .3cm

\item $C(x)\subset \mathbf{B}_h$, for every $x\in \mathbf{B}_h$, where  $C(x)\in \mathbf{A}(h)$ is the corresponding acyclic component containing $x$.

\end{enumerate}

\end{prop}

\begin{proof} The proof of (1) is very similar to the proof of the previous claim and we leave it to the reader. To prove (2) we recall the set $\mathbf{R}(h)$ from Theorem \ref{minimal-1}.  Since the minimal decomposition of  $\mathcal{E}(h)$ is finer than the one of $\mathcal{E}(\langle T_c^{-6},h\rangle)$, it follows that $\mathbf{M}(c,h)\subset \mathbf{R}(h)$. Together with  $\mathbf{B}_h\subset \mathbf{M}(c,h)$, this yields (2).

It remains to prove (3). Let $x\in \mathbf{B_h}$. Then by  (2) of this proposition we know there exists  $C(x)\in \mathbf{A}(h)$  containing $x$.
We argue by contradiction. Suppose that $C(x)$ is not contained in $\mathbf{B}_h$. Then there exists $y\in C(x)$ such that $y\in \mathbf{M}(c,h) - \mathbf{B}_h$. But then $y$ belongs to an acyclic component $D(y)\in \mathbf{A}(c,h)$. However, the minimal decomposition of  $\mathcal{E}(h)$ is finer than the one of $\mathcal{E}(\langle T_c^{-6},h\rangle)$, which  implies that $C(x)=C(y)\subset D(y)$. This means that $x\in D(y)$, and thus $x\in \mathbf{A}(c,h)$.
But by the definition we know that $\mathbf{B}_h\cap \mathbf{A}(c,h)=\emptyset$. This contradiction proves the proposition.

\end{proof}

\section{The proof of  Theorem \ref{main}}
In this section, we prove Theorem \ref{main} which states that the natural projection $p_g^a: \Homeo_+^a(S_g)\to \Mod(S_g)$ has no section over  ${\mathcal I}(S_g)$. Again, assume that $\mathcal{E}:{\mathcal I}(S_g)\to  \Homeo_+^a(S_g)$ is a section of $p_g^a$. Equip $S_g$ with a Riemann surface structure.
\subsection{Outline of the proof}
Recall that $c$ is a separating simple closed curve that divides the surface $S_g$ into a genus $4$ subsurface and a genus $g-4$ subsurface.  Fix a characteristic annulus $A$. Let $E_r$ be the set of points in $A$ that have rotation numbers equal to $r$ under $\mathcal{E}(T^{-6}_c)$. Lemma \ref{gap} states that the set $E_r$ is not empty when $0<r<1$.

The key observation of the proof lies in the analysis of connected components of $E_r$. Let $E$ be a component of $E_r$. We show the following:

\begin{enumerate}

\item $E$ is $\mathcal{E}(h)$-invariant for every simple BP map $h$

\vskip .3cm

\item $\overline{E}$ is a separator in $A$,

\vskip .3cm

\item if $E$ contains a periodic orbit, then $E$ contains  a separator. 

\end{enumerate}

Denote by $K(\overline{E})$  the annular completion of $\overline{E}$, and  let $\rho(\mathcal{E}(T^{-6}_c), K(\overline{E}))$ be the rotation interval of $K(\overline{E})$. We claim that $\rho(K(\overline{E}))=\{r\}$.  First of all, we know that  $r\in \rho(\mathcal{E}(T^{-6}_c), K(\overline{E}))$. If $\rho(\mathcal{E}(T^{-6}_c), K(\overline{E}))\neq \{r\}$, then $\rho(\mathcal{E}(T^{-6}_c), K(\overline{E}))$ contains infinitely many rational numbers. By Theorem \ref{PB}, there exist three periodic points $x_1,x_2,x_3\in K(\overline{E})$ with different rational rotation numbers $r_1,r_2,r_3$. Let $F_i$ denote the connected component of $E_{r_{i}}$ containing $r_i$, and let $M_i\subset F_i$ be a separator. 

By Proposition \ref{ordering}, there is an ordering on disjoint separators. Without loss of generality, we assume that $M_1<M_2<M_3$. Based on a discussion about the position $E$ with respect to $M_i$'s, we obtain a contradiction.
Thus,  $\rho(\mathcal{E}(T^{-6}_c), K(E))$ is the singleton $\{r\}$. 

We know from Theorem \ref{Matsumoto} that the left and right prime ends rotation numbers of $K(\overline{E})$ are both $r$. But in the group of circle homeomorphisms, the centralizer of an irrational rotation is essentially an abelian group. This contradicts the fact that a power of $\mathcal{E}(T^{-6}_c)$ is a product of commutators in its centralizer as in Proposition \ref{Johnson}.

\subsection{The set $E_r$} Once again we use abbreviation $f=\mathcal{E}(T^{-6}_c)$. For a characteristic annulus $A$, we let 
$$
E_r=\{x\in A: \rho\big(f,x,A\big)=r\}. 
$$
By Lemma \ref{notdepend}, we know that the definition of $E_r$ does not depend on the choice of the characteristic annulus. By Lemma \ref{gap}, if $0<r<1$, we know that $E_r$ is nonempty and $E_r\subset \mathbf{B}$.

\vskip .3cm

Next, we prove the following key lemmas.

\begin{lem}\label{invariant}
Fix $0<r<1$, and let $E$ denote a connected component  of $E_r$. Fix a simple BP map $h$. For $x\in E$, let $C(x) \in \mathbf{A}(h)$ be the corresponding acyclic set. Then $C(x)\subset E$.  In particular,  $E$  is ${\mathcal L\mathcal I}(c)$-invariant.
\end{lem}
\begin{proof}

To prove that $E$ is ${\mathcal L\mathcal I}(c)$ invariant, we only need to show that $E$ is invariant under simple BP maps since simple BP maps generate ${\mathcal L\mathcal I}(c)$ by Proposition \ref{Johnson}. Let $h$ denote a fixed simple BP map.

Recall the special characteristic annulus $A_h$ and the set $\mathbf{B_h}$ that we defined in Section \ref{ahbh}. Then, for each $x\in \mathbf{B_h}$ we have $C(x)\subset \mathbf{B_h}$, where $C(x)\in \mathbf{A}(h)$ is the corresponding acyclic set (we can do this by Proposition \ref{prop-novo}).

\begin{claim}\label{uniformbound}
There exists $d_0>0$ with the following properties. For  $x\in \mathbf{B_h}$, we let $C(x)\in \mathbf{A}(h)$ denote the corresponding acyclic set. Then every connected component of   the lift $\pi^{-1}(C(x))$ has the diameter at most $d_0$, for every  $x \in \mathbf{B_h}$ (the diameter is computed with respect to the Euclidean metric on the infinite strip $\widetilde{A_h}$).
\end{claim}
\begin{proof}  Let $\mathbf{d}:\mathbf{B_h} \to \mathbb{R}$ be the function such that $\mathbf{d}(x)$ is the diameter of a connected component of $\pi^{-1}(C(x))$. This definition does not depend on the choice of the connected component of  $\pi^{-1}(C(x))$ because different components are images of each other by the deck group of translations (and they are isometries for the Euclidean metric on $\widetilde{A_h}$).

Moreover, from the upper-semicontinuity of the acyclic decomposition $\mathbf{A}(h)$, it follows that  $\mathbf{d}$ is an upper-semicontinuous function. Thus,  it achieves its  maximum on the  compact set $\mathbf{B}_h$. We let $d_0$ be the maximum value of the function $\mathbf{d}$.
\end{proof}

For  $x\in E$, It remains to show that $C(x)\subset E$. Since $C(x)$ is compact and connected, and $C(x)\subset \mathbf{B_h}$, it suffices to show that the rotation number of each $y \in C(x)$ is equal to $r$.

Fix  $\widetilde{f}:\widetilde{A_h}\to \widetilde{A_h}$, a lift of $f$, and $\widetilde{x} \in \pi^{-1}(x)$. Let $\widetilde{y} \in \pi^{-1}(y)$ be the point which belongs to the  same connected component of $\pi^{-1}(C(x))$ as $\widetilde{x}$. We denote this connected component of $\pi^{-1}(C(x))$ by $D$. Since $f$ permutes the
acyclic sets $C(z)$, for $z\in \mathbf{B_h}$, from the previous claim we conclude that the Euclidean distance between $\widetilde{f}^k(\widetilde{x})$ and $\widetilde{f}^k(\widetilde{y})$ is at most $d_0$, for any integer $k$. Therefore, the translation numbers $\rho(\widetilde{f},\widetilde{x},\widetilde{A})$ and $\rho(\widetilde{f},\widetilde{y},\widetilde{A})$ are equal (since $x\in E\subset E_r$ we already know that $\rho(\widetilde{f},\widetilde{x},\widetilde{A})$  exists). Thus, $y\in E$, and we are done.

\end{proof}

\vskip .3cm

\subsection{Further properties of connected components of $E_r$} In this subsection, we show that the closure of  each connected component of $E_r$ is a separator when $0<r<1$. Let $E$ be one connected component of $E_r$. Fix any characteristic annulus $A$. Denote by $\pi: \widetilde{A}\to A$ the universal cover and recall  the $x$-coordinate function $p_1: \widetilde{A}\to \mathbb{R}$, on the infinite strip $\widetilde{A}$. (The function $p_1$ is what we use to define translation numbers on $\widetilde{A}$.)

\begin{lem} \label{separatingL0}
The closed set $\overline{E}$ is a separator (as defined in Section 2). 
\end{lem}
\begin{proof}
By Lemma \ref{Johnson}, the left Torelli group ${\mathcal L\mathcal I}(c)$ is generated by simple BP maps. Write $T_c^{-6}$ as the product of simple BP maps
$$
T_c^{-6}=h_k\cdots h_1.
$$
For simplicity, we let $g_i=\mathcal{E}(h_i)$. Then 
\begin{equation}\label{eq-ajmo-1}
f=g_k\cdots g_1.
\end{equation}

\vskip .3cm

Fir $x_1\in E$, and  let $C_1$ be the element of $\mathbf{A}(h_1)$ that contains $x_1$.  Inductively for  $i \in \mathbb{Z}$, we let $C_{i+1}$ be the element of $\mathbf{A}(h_{i+1})$ which contains $x_{i+1}$, where $x_{i+1}=g_i(x_i)$.  By Lemma \ref{invariant} we have $x_i\in E$. Moreover, from (\ref{eq-ajmo-1}) we find that 
\begin{equation}\label{eq-ajmo-2}
x_{i+k}=f(x_i),
\end{equation}
for every $i\in \mathbb{Z}$.

Now, we lift everything to the universal cover $\pi:\widetilde{A}\to A$. The corresponding lift of $C_i$ is also denoted by $C_i$, while the corresponding lift of the point $x_i$ is denoted by $\widetilde{x}_i$. Once we fix a lift $\widetilde{x}_1$ of $x$, the remaining lifts are uniquely determined. Moreover, there exists a unique lift 
$\widetilde{f}:\widetilde{A}\to \widetilde{A}$ such that $\widetilde{f}(\widetilde{x}_i)=\widetilde{x}_{i+k}$.

\begin{figure}[H]
\includegraphics[scale=0.4]{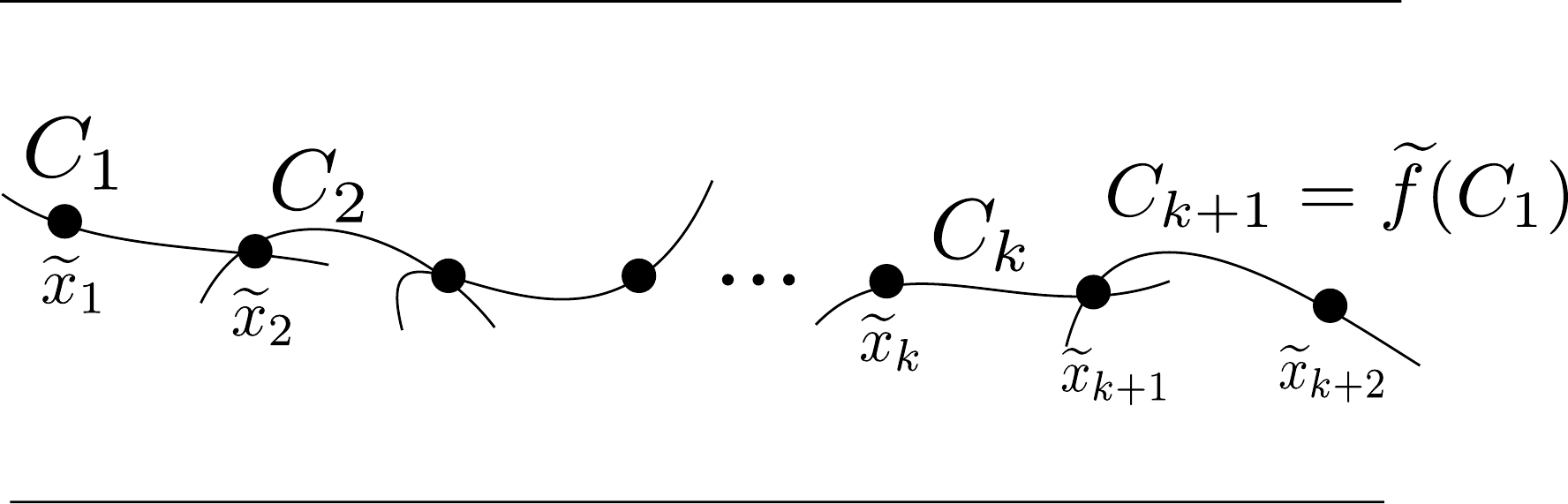}
\caption{}
\label{longline}
\end{figure}
The sequence of subsets $C_{i}$ satisfies the following properties as in Figure \ref{longline}: 

\vskip .3cm

\begin{itemize}
\item $C_i$ and $C_{i+1}$ contain a common point for every $i\in \mathbb{Z}$ (by definition);
\item $C_i$ is connected and $C_i\subset \pi^{-1}(E)$ (by Lemma \ref{invariant}).
\end{itemize}

\vskip .3cm

Define $$
C=C_1\cup...\cup C_k,
$$ 
and 
$$
K_0:=\bigcup_{n\in \mathbb{Z}}\widetilde{f}^n(C).
$$

\vskip .3cm

We now prove that $\overline{E}$ is a separator. Since $E$ is connected (and compactly contained in $A$), it follows that $\overline{E}$ is compact and connected in $A$. It remains to show it separates the two ends of $A$.  Let $\gamma$ denotes a simple closed arc  in $\overline{E}$ connecting the two ends of $A$ (we choose $\gamma$ so it  connects two accessible points of the two frontiers of $A\subset S_g$). 
It suffices to show that $\overline{E}$ intersects any such $\gamma$. In fact, we show a stronger statement that $E$ intersects any such $\gamma$.

We argue by contradiction, and assume that there is such a $\gamma$ which $E$ does not intersect.  Choose a  lift $\widetilde{\gamma}$ of $\gamma$. Then  $\widetilde{\gamma}$ divides $\widetilde{A}$ into two sides $\Omega_-$ and $\Omega_+$, such that  $p_1(\Omega_-)$ is bounded above and $p_1(\Omega_+)$ is bounded below (recall that $p_1$ is the $x$-coordinate map). Then $\pi^{-1}(E)$ does not intersect $\widetilde{\gamma}$. In the rest of the proof we show that $K_0$ intersects $\widetilde{\gamma}$, which is a contradiction.

Since the translation number $\rho\big(\widetilde{f},\widetilde{x}_1,\widetilde{A}\big)$ is not equal to zero,  we conclude that 
$$
\lim_{n\to +\infty} \, p_1\big(f^{n}(\widetilde{x}_1)\big)=+\infty,
$$
and 
$$
\lim_{n\to +\infty} \, p_1\big(f^{n}(\widetilde{x}_1)\big)=+\infty,
$$
Together with  (\ref{eq-ajmo-2}), this implies  that $K_0$ intersects both $\Omega_+$ and $\Omega_-$.  Define a function $H:\Omega_+\cup \Omega_-\to \mathbb{R}$ by letting $H(x)=-1$ for $x\in \Omega_-$ and $H(x)=1$ for $x\in \Omega_+$. Then $H$ is a continuous function on $\Omega_+\cup \Omega_-$. If we assume that $K_0$ does not intersect $\widetilde{\gamma}$, then $K_0\subset \Omega_+\cup \Omega_-$ and the restriction of $H$ to $K_0$ is continuous well.
However, $K_0$ is connected so $H(K_0)$ is a connected subset of $\mathbb{R}$. But this  $H(K_0)-\{0,1\}$, which is not connected. It follows that $K_0$ intersects $\widetilde{\gamma}$ and we are finished. 
\end{proof}

\vskip .3cm

Building on the construction from the previous proof we show that for a connected component $E$ that contains a periodic orbit  the following stronger property holds.

\begin{lem}\label{rational}
Let $x$ be a periodic orbit of $f$ such that $\rho(f,x,A)=p/q$ and $0<p/q<1$. Then, the connected component $E$ of $E_{p/q}$ which contains $x$, also contains a separator (as a subset).
\end{lem}
\begin{proof}
By construction of the set $K_0$, we have
$$
\widetilde{f}(K_0)=K_0.
$$ 
Since $x=x_1$ is a periodic point for $f$, there exists an integer $l$ such that $f^{l}(x_1)=x_1$. Then by (\ref{eq-ajmo-2})  we have $x_{1+kl}=x_1$.
This implies that for some   integer $m$ the equality 
$$
\widetilde{x}_{1+kl}=T^m(\widetilde{x}_1),
$$
holds, where $T^m$ is the translation by $m$.

Thus, $K_0$ is invariant under $T^m$. This shows that $\pi(K_0)\subset E$ is compact. Furthermore $\pi(K_0)$ is connected since $K_0$ is connected, and we proved in the previous lemma that $\pi(K_0)$ separates the ends of $A$. Thus, $\pi(K_0)$ is a separator.
\end{proof}

\subsection{Finishing the proof}
Fix an irrational number $r\in (0,1)$. By Lemma \ref{gap}, we know that $E_r$ is not empty. Let $E$ be a connected component of $E_r$. 
By Lemma \ref{invariant}, we know that $E$ is invariant under ${\mathcal L\mathcal I}(c)$. By Lemma \ref{separatingL0}, we know that $\overline{E}$ is a separator. The annular completions $K(\overline{E})$ of $\overline{E}$ is also ${\mathcal L\mathcal I}(c)$-invariant since the definition is canonical. 
The following claim is at the heart  of the entire construction. 
\begin{claim}
Let $r_L$ and $r_R$ be the left and right prime ends rotation numbers of $f$ on $K(\overline{E})$. Then $r_L=r_R=r$. 
\end{claim}
\begin{proof}
We prove that the rotation interval $\sigma(f,K(\overline{E}))$ is a singleton $\{r\}$. Then by Theorem \ref{Matsumoto}, we know that $r_L=r_R=r$.

Since $K(\overline{E})$ is an annular continuum, and $f$ is area-preserving, we have Theorem \ref{PB} saying that every rational number in the translation interval $\sigma(f,K(\overline{E}))$ is realized by a periodic orbit of $f$. We argue by contradiction. Suppose $\sigma(f,K(\overline{E}))$ is not a singleton, there exist periodic points $x_1,x_2,x_3\in K(\overline{E})$ with three different rotation numbers  $r_i \in \sigma(f,K(\overline{E}))$. Denote by $M_i$ a separator contained the connected component of $E_{r_{i}}$ containing $x_i$ (such $M_i$ exists by Lemma \ref{rational}). Without loss of generality, we assume that $M_1<M_2<M_3$ by Proposition \ref{ordering}. 

Since $E$ consists of points with irrational rotation number $r$, we know that $E$ is disjoint from the separators $M_1,M_2$, and $M_3$. On the other hand, each $M_i$ is contained in $K(\overline{E})$. We show this yields a  contradiction.  We break the discussion into the following two cases. (Recall that for an annular continuum $K\subset A$, by $A_R$ and $A_L$ we denote the two annuli in the complement of $K$.)

\vskip .3cm
Since $E$ is disjoint from $M_2$ and $A_R(M_2)$ is a connected component of $A-M_2$, one of the following must happen.
\begin{itemize}

\item $E\cap A_R(M_2)=\emptyset$: We claim that $K(\overline{E}) \cap A_R(M_2)=\emptyset$ which contradicts $x_3\in A_R(M_2)\cap K(\overline{E})$. Since $E\cap A_R(M_2)=\emptyset$ and that $A_R(M_2)$ is open, we know that $\overline{E}\cap A_R(M_2)=\emptyset$. Since $A_R(M_2)$ is connected, disjoint from $\overline{E}$ and contains the left end of $A$, we know $A_R(M_2) \subset A_R(\overline{E})$  by Proposition \ref{ordering}. Therefore $A_R(M_2)\cap K(\overline{E})=\emptyset$.
 
\vskip .3cm

\item $E\subset A_R(M_2)$ which means $E\cap A_L(M_2)=\emptyset$: With the same argument above, we show that $K(\overline{E})\cap A_L(M_2)=\emptyset$ which contradicts $x_1\in A_L(M_2)\cap K(\overline{E})$.
\end{itemize}

\end{proof}

We conclude that ${\mathcal L\mathcal I}(c)$ acts on the left prime ends of $K(\overline{E})$, where the action of $f$ has irrational rotation number $r$. However, $f$ is a product of commutators in ${\mathcal L\mathcal I}(c)$ by Lemma \ref{Johnson}, the following lemma gives us a contradiction.

\begin{lem}[Centralizer of an irrational rotation]
If $\phi\in \Homeo_+(S^1)$ has an irrational rotation number, then $\phi$ cannot be written as a product of commutators in its centralizer.
\end{lem}
\begin{proof}
Since the rotation number of $\phi$ is irrational, $\phi$ has no periodic orbit. Let  $\mathbf{M}\subset S^1$ denote the minimal set of $\phi$  (in particular, the orbits of points in $\mathbf{M}$ are dense in $\mathbf{M}$ by \cite[Proposition 5.6]{Ghys}). Then either $\mathbf{M}$ is equal to $S^1$ or it is a Cantor set. In the latter case, the complement of $\mathbf{M}$ is a countable union of open interval. Collapsing these intervals to points we obtain the quotient space $S^1/\sim$ which is homeomorphic to $S^1$. Note that every homeomorphism which belongs to the centralizer of $\phi$ descends to a homeomorphism of the quotient.

Therefore, we may assume that the minimal set of $\phi$ is the circle.  Then by a theorem of Poincar\'e (see e.g., \cite[Theorem 5.9]{Ghys}), we know 
$\phi$  is conjugate to an actual irrational rotation. It suffices to show that an irrational rotation $\phi$ can not be written as  a product of commutators in its centralizer. However the centralizer of $\phi$ is the Abelian group $SO(2)$, thus any commutator in the centralizer of $\phi$ is the identity map of $S^1$.
The proof is complete.

\end{proof}

\section{Pseudo-Anosov analysis and proof of Theorem \ref{minimal} and \ref{minimal-1}}
Let $b\in S_2$ be a base-point. Let $\mathcal{Z}:S_2\to S_2$ be a pseudo-Anosov map on $S_2$ such that $b$ is one singularity. We can ``blow up" the base-point $b$ to a circle. Let $M$ be a surface of genus $g\ge 2$. We decompose $M$ into the union of a genus $2$ surface missing one disk $L$, a closed annulus $N$ and a genus $g-2$ surface missing one disk $R$. We construct a map $\mathcal{Z}_M:M\to M$ as the following: $\mathcal{Z}_M$ is the identity map on $R$, the blow up of $\mathcal{Z}$ on $L$ and any action on $N$. 

Let $\mathcal{P}: M\to S_2$ be the map that collapses points in $N\cup R$ to a point. Then let $\widetilde{M}$ be the cover of $M$ that is a pull back of the universal cover $\mathbb{H}^2\to S_2$ where $\mathbb{H}^2$ denotes the hyperbolic plane. We pick a lift of the homeomorphism $\mathcal{Z}$ to the universal cover $\widetilde{\mathcal{Z}}:\mathbb{H}^2\to \mathbb{H}^2$. There is a projection $\widetilde{\mathcal{P}}:\widetilde{M}\to \mathbb{H}^2$. Geometrically it is the pinching map that pinches each copy of lifts of $R$ on $\widetilde{M}$. The map $\mathcal{Z}_M$ can also be lifted to $\widetilde{M}$ as $\widetilde{\mathcal{Z}_M}:\widetilde{M}\to \widetilde{M}$.

Let $\mathcal{F}$ be a homeomorphism that is homotopic to $\mathcal{Z}_M$. Since $\mathcal{F}$ and $\mathcal{Z}_M$ are homotopic, we could lift $\mathcal{F}$ to $\widetilde{\mathcal{F}}: \widetilde{M}\to \widetilde{M}$ such that $\widetilde{\mathcal{F}}$ and $\widetilde{\mathcal{Z}_M}$ have bounded distance.
\begin{defn}
For $\widetilde{x}\in \widetilde{M}$ and $\widetilde{y}\in \mathbb{H}^2$, we say that $(\widetilde{\mathcal{F}},\widetilde{x})$ \emph{shadows} $(\widetilde{\mathcal{Z}},\widetilde{y})$ if there exists $C$ such that
\[
d_{\mathbb{H}^2}(\widetilde{\mathcal{P}}(\widetilde{\mathcal{F}}^n(\widetilde{x})),\widetilde{\mathcal{Z}}^n(\widetilde{y}))<C
\]
\end{defn}
We call that a sequence of points $\{x_n\}$ in $\mathbb{H}^2$ is a an $\widetilde{\mathcal{Z}}$ pseudo-orbit if the set $\{d_{\mathbb{H}^2}(\widetilde{\mathcal{Z}}(x_n),x_{n+1})\}$ is bounded. \begin{lem}
The sequence $\{\widetilde{\mathcal{P}}(\widetilde{\mathcal{F}}^n(x))\}$ is an $\widetilde{\mathcal{Z}}$-pseudo-orbit for every $x\in \widetilde{M}$.
\end{lem}
\begin{proof}
This lemma follows from the following inequality:
\[
d_{\mathbb{H}^2}(\widetilde{\mathcal{Z}}(\widetilde{\mathcal{P}}(\widetilde{\mathcal{F}}^n(\widetilde{x}))),\widetilde{\mathcal{P}}(\widetilde{\mathcal{F}}^{n+1}(\widetilde{x})))=d_{\mathbb{H}^2}(\widetilde{\mathcal{P}}(\widetilde{\mathcal{Z}_M}(\widetilde{\mathcal{F}}^n(\widetilde{x}))),\widetilde{\mathcal{P}}(\widetilde{\mathcal{F}}(\widetilde{\mathcal{F}}^{n}(\widetilde{x}))))\le C.\qedhere
\]
\end{proof}
There is a difference between a pseudo-Anosov map and an Anosov map: for an Anosov homeomorphism, every pseudo-orbit has a uniformly bounded distance to a unique actual orbit; however for a pseudo-Anosov homeomorphism, we do not have such nice relation. The stable and unstable foliations have singularities and their leaf spaces $L^s$ and $L^u$ are more complicated. Namely, $L^s$ and $L^u$, and their metric completions $\overline{L^s}$ and $\overline{L^u}$ have the structure of $\mathbb{R}$-trees \cite{pA}. Then $\widetilde{\mathcal{Z}}$ induces a map $\widetilde{\mathcal{Z}}^s: \overline{L^s}\to \overline{L^s}$ that uniformly expands distance by a factor $\lambda>1$ and a map $\widetilde{\mathcal{Z}}^u: \overline{L^u}\to \overline{L^u}$ that uniformly contracts distance by a factor $1/\lambda$. There is an embedding of $\mathbb{H}^2$ in $\overline{L^s}\times \overline{L^u}$. Denote by $\overline{\mathcal{Z}}=(\widetilde{\mathcal{Z}}^s,\widetilde{\mathcal{Z}}^u)$. As discussed in \cite{pA}, every $\widetilde{\mathcal{Z}}$ pseudo-orbit has a uniformly bounded distance to a unique actual orbit of $\overline{\mathcal{Z}}$ on $\overline{L^s}\times \overline{L^u}$.

Using this property, there exists a unique map $\Theta: \widetilde{M}\to \overline{L^s}\times \overline{L^u}$ such that $\{\widetilde{\mathcal{P}}(\widetilde{\mathcal{F}}^n(x))\}$ is shadowed by the orbit of $\Theta(x)$. We have the following two theorems from \cite[Lemma 4.14]{Mar} and \cite[Theorem 1.2]{pA}.
\begin{thm}
Let $\mathcal{F},\mathcal{G}\in \Homeo(M)$ such that $\mathcal{F}, \mathcal{G}$ commute, $\mathcal{G}$ is homotopic to the identity map on the component $L$ and $\mathcal{F}$ is homotopic to $\mathcal{Z}_M$. We have that $\mathcal{G}$ preserves each connected component of $\Theta^{-1}(c,w)$ for $(c,w)\in \overline{L^s}\times \overline{L^u}$.
\end{thm}
Let $\widetilde{\mathbf{S}}$ be the collection of all components of the sets $\Theta^{-1}(c,w)$. Set $\mathbf{S}=\pi_M(\widetilde{\mathbf{S}})$ where $\pi_M:\widetilde{M}\to M$ be the covering map. The following is \cite[Proposition 4.1]{Mar}.
\begin{prop}
The set $\mathbf{S}$ is a proper upper semi-continuous decomposition of $M$. Moreover, there exists a simple closed curve $\gamma$, which is homotopic to the boundary of  $L$ (the surface of genus $2$ minus a disc) such that if $p\in M$ belongs to the component of $M-\gamma$ that is homotopic to $L$, then the component of $\mathbf{S}$ that contains $p$ is acyclic.
\end{prop}
Let $\mathcal{E}: {\mathcal I}(M)\to \Homeo_+(M)$ be a section of $p_g$. Using the above ingredients, we can prove the following lemma the same way as \cite[Theorem 4.1]{Mar} by the existence of pseudo-Anosov elements in ${\mathcal I}(S_2)$ (see, e.g., \cite[Corollary 14.3]{FM}).
\begin{thm}\label{bigacyclic}
Let $M$ be a surface of genus $g>2$ and $\alpha\subset M$ be a simple closed curve such that $\alpha$ separates $M$ into a genus $2$ surface minus a disc and a compact surface of genus $g-2$ with one end. For an element $h\in {\mathcal I}(M)$ that can be realized as a homeomorphism that is the identity inside the corresponding subsurface of $M$ that is homeomorphic to a genus $2$ surface minus a disc, there exists an admissible decomposition of $M$ for $\mathcal{E}(h)$ with the following property: there exists a simple closed curve $\beta$, homotopic to $\alpha$ such that if $p\in M$ belongs to the genus $2$ surface minus a disc (which is one of the two components obtained after removing $\beta$ from $M$), then the component of the decomposition that contains $p$ is acyclic. 
\end{thm}
We now use the above to proof Theorem \ref{minimal} and \ref{minimal-1}.
\begin{proof}[Proof of Theorem \ref{minimal} and \ref{minimal-1}]
For simplicity, we prove the existence of $\mathbf{M}_1(c,h)$ only. The proof of others are similar. Define  
\[H=\langle \mathcal{E}(T_c^{-6}),\mathcal{E}(h)\rangle<\Homeo(M)\] for $h=T_aT_b^{-1}$. The other cases can be proved the same way. Let $\mathbf{S}(H)$ be the minimal decomposition of $H$ and $\mathbf{A}(H)$ be the union of acyclic components of $\mathbf{S}(H)$, which is a subsurface by Proposition \ref{subsurface}. Let $\alpha$ be the curve as in the following figure.
\begin{figure}[H]
 \labellist 
  \small\hair 2pt
     \pinlabel $a$ at 80 80
          \pinlabel $b$ at 80 30
                    \pinlabel $c$ at 200 50
                          \pinlabel $\alpha$ at 180 100
      \endlabellist
     \centerline{ \mbox{
 \includegraphics[scale=0.5]{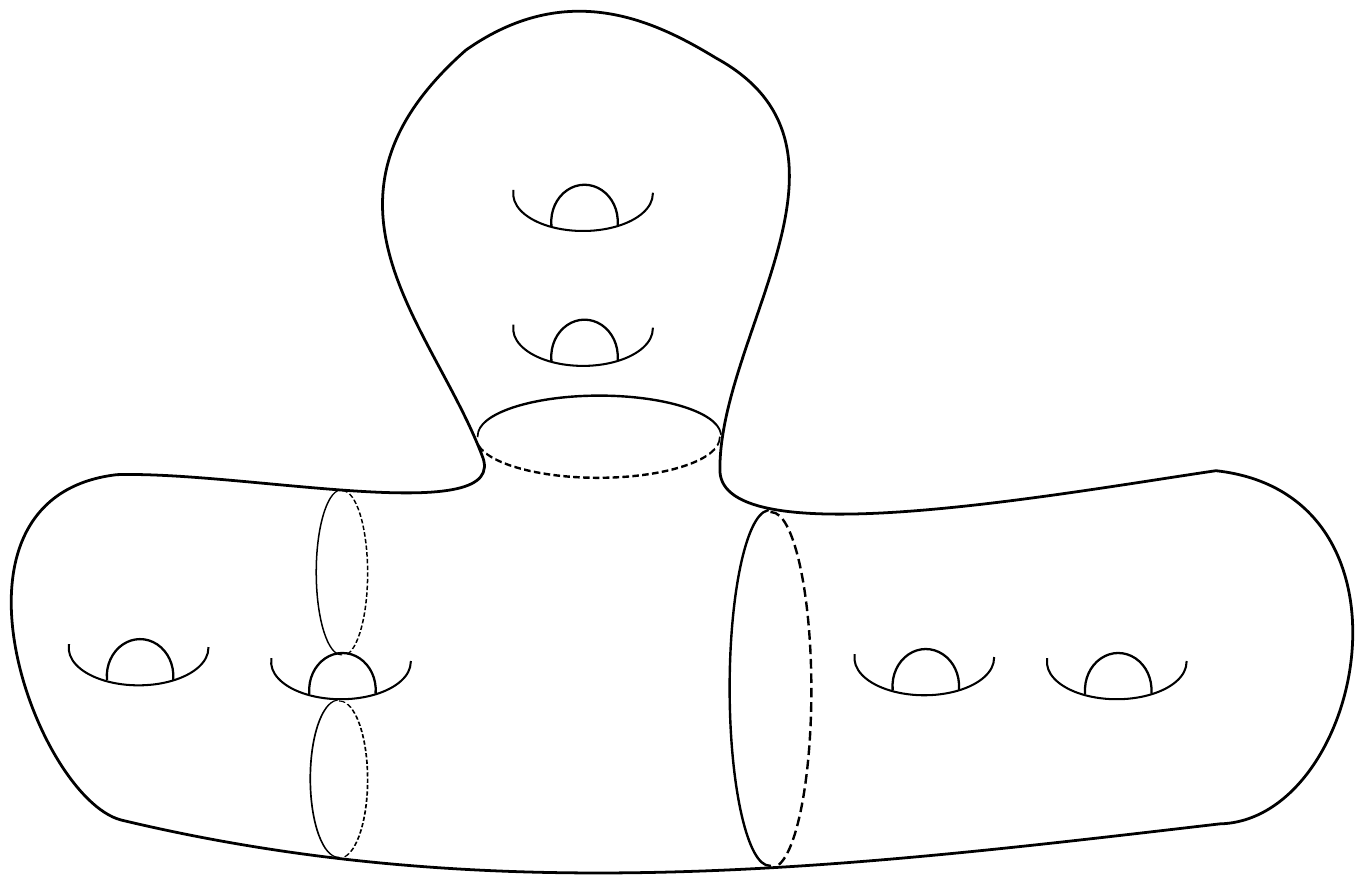}}}
 \caption{}
\end{figure}
By Theorem \ref{bigacyclic}, there exists a simple closed curve $\beta$ homotopic to $\alpha\subset M$ such that there exists a component $W$ of $\mathbf{A}(H)$ that contains the component of $M-\beta$ that is homotopic to the genus $2$ surface minus a disc. Since $H$ is not homotopic to the identity map, we know that $W\neq M$. This implies that $W$ has at least one end. We plan to prove that $W$ has exactly three ends and they are homotopic to $a,b,c$ respectively. 

Let $\beta_n$ be a nested sequence that determines one end $E$ of $W$. Firstly, we claim that $\beta_n$ cannot intersect $a,b,c$ (as isotopy classes of simple closed curves). If $\beta_n$ intersects $a$ geometrically, then $\mathcal{E}(h)(\beta_n)$ intersects $\beta_n$. This contradicts the fact that $\mathcal{E}(h)(\beta_n)\subset W$. If $\beta_n$ is not homotopic to $a,b$ or $c$, then there exists a separating curve $\delta$ such that $\delta$ intersects $\beta_n$ and $\delta$ does not intersect $a,b,c$. Then since $\mathcal{E}(T_\delta)(\beta_n)$ intersect $\beta$, it has to intersect $E$. However since $\mathcal{E}(T_\delta)$ commutes with $H$, we know that $\mathcal{E}(T_\delta)$ permutes components of $\mathbf{A}(H)$. This contradicts the fact that $\mathcal{E}(T_\delta)(\beta_n)$ intersect $\beta$ and that $\beta$ is a nested sequence that determines one end of $W$.

We now need to show that $a,b,c$ are all homotopic to some end of $W$. We prove this by contradiction. If the frontier of $W$ does not contain one end homotopic to $a$, then $W$ contains at most two ends, homotopic to $b,c$. Let $p_W: W\to W/\sim$ be the Moore map for components of $\mathbf{S}(H)$ in $W$. Let $A_b$ be an open annulus that is bounded by the end of $W/\sim$ homotopic to $b$ and a simple closed curve homotopic to $b$. We define $A_c$ similarly. Then $U=p_W^{-1}(A_b)\cup (M-W)\cup p_W^{-1}(A_c)$ is an open set. We define a new upper semi continuous decomposition $\mathbf{S}'$ that consists of elements of $\mathbf{S}(H)$ in $M-U$ and points in $U$. Let $p: M\to M/\sim $ be the Moore map for $\mathbf{S}'$. Therefore $H$ is semi-conjugate to a new action $H'$ that is the identity on $M-U$. This contradicts the fact that $\mathcal{E}(h)$ is homotopic to $T_aT_b^{-1}$ in $M-U$, which is not homotopic to identity. The proof that $b,c$ are also homotopic to ends of $W$ is similar.
\end{proof}

    	\bibliography{torelli}{}
	\bibliographystyle{alpha}

\end{document}